\documentclass[12pt]{article}
\usepackage{titlesec}
\titleformat{\subsection}{\normalsize\itshape}{\thesubsection}{1em}{}
\usepackage{amssymb, amsmath, amsthm, amsfonts, enumerate,bm,nicefrac}
\usepackage{amsmath,setspace,scalefnt}
\usepackage[usenames,dvipsnames,svgnames,table]{xcolor}
\usepackage{graphicx,tikz,caption,subcaption}
\usetikzlibrary{patterns}
\usetikzlibrary{shapes}
\usepackage{fullpage}
\usepackage{hyperref}
\usepackage{enumitem,verbatim}
\usepackage{multicol}
\usepackage{float}
\usepackage{cleveref}
\usepackage{bm}

\usepackage{indentfirst}
\providecommand{\bibinfo}[2]{#2}

\usepackage[margin=2.5cm]{geometry}

\definecolor{gray}{rgb}{0.25, 0.25, 0.25}
\usepackage{tikz}
\usepackage{booktabs}

\newtheorem{theorem}{Theorem}[section]
\newtheorem{lemma}[theorem]{Lemma}
\newtheorem{claim}{Claim}

\newtheorem*{observation*}{Observation}

\newtheorem*{problem*}{Problem}

\newtheorem*{question*}{Question}

\theoremstyle{definition}
\newtheorem{definition}[theorem]{Definition}

\newtheorem{fact}{Fact}
\newenvironment{definition*}
  {
   \innerdefinition}
  {\endinnerdefinition}

\theoremstyle{remark}

\usepackage{todonotes}

\title{\Large\textbf{The signless Laplacian spectral Tur\'an problems\\ for hypergraphs}\thanks{Research was partially supported by the National
		Nature Science Foundation of China (grant numbers 12331012, 12571375)} }

\author{}
\date{}
\author {Yongchun Lu$^{1}$,     Jiadong Wu$^{1}$,  Liying Kang$^{1,2}$\thanks{\em Corresponding author. Email address: lykang@shu.edu.cn (L. Kang), luyongchun@shu.edu.cn (Y. Lu), 1753381890@qq.com (J. Wu)} \\
	{\small $^{1}$ Department of Mathematics, Shanghai University,
		Shanghai 200444, P.R. China}\\
	{\small$^{2}$Newtouch Center for Mathematics of Shanghai University,
		Shanghai,  China, 200444}}

\begin{document}
\maketitle

\begin{abstract}
Let $\mathcal{H}=(V, E)$ be an $r$-uniform hypergraph on $n$ vertices. The signless Laplacian spectral radius  of $\mathcal{H}$ is defined as  the maximum modulus of the eigenvalues of the tensor $\mathcal{Q}(\mathcal{H})=\mathcal{D}(\mathcal{H})+\mathcal{A}(\mathcal{H})$, where $\mathcal{D}(\mathcal{H})$ and $\mathcal{A}(\mathcal{H})$ are the  degree diagonal tensor and the adjacency tensor of $\mathcal{H}$, respectively. In this paper, we establish a general theorem that extends the spectral Tur\'an result of Keevash, Lenz and Mubayi
[SIAM J. Discrete Math., 28 (4) (2014)]
 to the setting of signless Laplacian spectral Tur\'an problems.
We prove that if a family $\mathcal{F}$ of $r$-uniform hypergraphs  is degree-stable with respect to a family $\mathcal{H}_n$ of $r$-uniform hypergraphs and its extremal constructions satisfy certain natural assumptions, then the signless Laplacian spectral Tur\'an problem for $\mathcal{F}$  can be effectively reduced to
the corresponding problem restricted to the family $\mathcal{H}_n$.  As a concrete application, we completely determine the extremal hypergraph that maximizes the signless Laplacian spectral radius among all Fano plane-free $3$-uniform hypergraphs, showing that the unique extremal hypergraph is the balanced complete bipartite $3$-uniform hypergraph.

\end{abstract}
\begin{flushleft}
		\hspace{2.5em}\textbf{Keywords:} Signless Laplacian spectral radius; Fano plane; Stability; Extremal hypergraph
	\end{flushleft}

\section{Introduction}


For an integer $r\ge2$, an \emph{$r$-uniform hypergraph} (or simply \emph{$r$-graph}) $\mathcal{H}$ is a pair $(V, E)$, where $V$ denotes a finite set of vertices and $E$ is a collection of $r$-subsets of $V$.    Given a family $\mathcal{F}$ of $r$-graphs, an $r$-graph  $\mathcal{H}$ is said  to be \emph{$\mathcal{F}$-free} if it  contains no member of  $\mathcal{F}$ as a subgraph. The classical Tur\'an problem seeks the maximum number of edges in an $\mathcal{F}$-free
$r$-graph on $n$ vertices. A core concept associated with this problem is the \emph{Tur\'an number}, denoted by $\operatorname{ex}_r(n, \mathcal{F})$, which is defined as the maximum number of edges in an $\mathcal{F}$-free $r$-graph on $n$ vertices. The \emph{Tur\'an density} $\pi(\mathcal{F})$ of $\mathcal{F}$ is defined as
$$\pi(\mathcal{F}) = \lim_{n \to \infty}\frac{\operatorname{ex}_r(n, \mathcal{F})}{{\binom{n}{r}}}.$$
An $r$-graph $\mathcal{H}$ is called \emph{$k$-colorable} if its vertex set $V(\mathcal{H})$ can be partitioned into $k$ subsets such that no subset contains an edge of $\mathcal{H}$.

 Let $\mathcal{H}$ be  an $r$-graph  on $n$ vertices.
For every vertex $v \in V(\mathcal{H})$, the \emph{link} of $v$ is defined as \scalebox{0.95}{$L_{\mathcal{H}}(v)=\{A \subseteq V(\mathcal{H}) \mid A \cup\{v\} \in E(\mathcal{H})\}$}, and the \emph{degree} of $v$ is $d_{\mathcal{H}}(v)=|L_\mathcal{H}(v)|$.
The \emph{minimum degree} of $\mathcal{H}$ is defined as \scalebox{0.95}{$\delta(\mathcal{H})=\min \{d_\mathcal{H}(v) \mid v\in V(\mathcal{H})\}$}.
The \emph{ordered link} of $v$, denoted by $L_{\mathcal{H}}^o(v)$, is defined as the set of all $(r-1)$-tuples $(u_1, \ldots, u_{r-1}) \in V(\mathcal{H})^{r-1}$ such that $\{u_1, \ldots, u_{r-1}, v\} \in E(\mathcal{H})$. Let $E_\mathcal{H}(v)$ denote the set of edges incident to $v$.
When the context is clear, we abbreviate these notations as $L(v)$, $d(v)$, $L^o(v)$, and $E_v$, respectively. Let $E_v^{r-1}$ denote the collection of all $(r-1)$-subsets of $V(\mathcal{H}) \setminus \{v\}$. Let $V_v^{r-1}$ denote the set of all $(r-1)$-tuples with entries in $V(\mathcal{H}) \setminus \{v\}$. Let $x^{\bm{\nu}}=x_{i_1}x_{i_2}\cdots x_{i_{r-1}}$ for an $(r-1)$-tuple $\bm{\nu}=(i_1, i_2, \ldots, i_{r-1})\in V_v^{r-1}$.
   For a graph $F$, let $F^{(r)}$ be an $r$-graph obtained from $F$ by enlarging each edge of $F$
with $(r - 2)$ new vertices disjoint from $V (F)$ such that distinct edges of $F$
are enlarged by distinct vertices.

Extremal graph theory constitutes a cornerstone of combinatorics, with Tur\'an-type problems forming one of its central themes. This research area traces its origins to Tur\'an's seminal theorem~\cite{turan1941extremal}, which asserts that $\operatorname{ex}(n, K_{k+1}) = |E(T_{n, k})|$. Here, the \emph{Tur\'an graph} $T_{n, k}$ is defined as the complete   $k$-partite graph on $n$   vertices where the partition classes are as equal in size as possible.


Motivated by the classical Tur\'an problem, the spectral Tur\'an problem has emerged as a topic of considerable interest in combinatorics. This problem investigates extremal graph structures through the lens of eigenvalues-most commonly those of the adjacency or signless Laplacian tensor (matrix). As a natural extension of the classical Tur\'an problem, a central research objective in this area is to derive spectral analogues of classic edge-extremal results. To introduce the relevant literature, we first give some basic concepts.

Let $\mathcal{A}$ be an order $r\geq 2$, dimension $n$ tensor, $\bm{x}=(x_1,x_2,\ldots,x_n)^{\mathrm{T}}$ be a column
vector of dimension $n$. Then $\mathcal{A}\bm{x}$ is a vector in $\mathbb{C}^n$, whose
$i$-{\em th} component is as the following
\[
(\mathcal{A}\bm{x})_i=\sum_{i_2,\ldots,i_r=1}^na_{ii_2\cdots i_r}x_{i_2}\cdots x_{i_r},~~
i\in [n]
\]
and
\[
\bm{x}^{\mathrm{T}}(\mathcal{A}\bm{x})=\sum_{i_1,i_2,\ldots,i_r=1}^n
a_{i_1i_2\cdots i_r}x_{i_1}x_{i_2}\cdots x_{i_r}.
\]

In 2005, Lim \cite{Lim} and Qi \cite{Qi2005} independently introduced the concepts of
tensor eigenvalues and the spectra of tensors. Let $\mathcal{A}$ be an order $r$ and
dimension $n$ tensor, $\bm{x}=(x_1,x_2\ldots,x_n)^{\mathrm{T}}\in\mathbb{C}^n$
be a column vector of dimension $n$. If there exists a number $\lambda\in\mathbb{C}$
and a nonzero vector $\bm{x}\in\mathbb{C}^{n}$ such that
\[
\mathcal{A}\bm{x}=\lambda \bm{x}^{[r-1]},
\]
then $\lambda$ is called an {\em eigenvalue} of $\mathcal{A}$, $\bm{x}$ is called
an {\em eigenvector} of $\mathcal{A}$ corresponding to the eigenvalue $\lambda$,
where $\bm{x}^{[r-1]}=(x_1^{r-1},x_2^{r-1},\ldots,x_n^{r-1})^{\mathrm{T}}$. The
{\em spectral radius} $\rho(\mathcal{A})$ of $\mathcal{A}$ is the maximum modulus
of the eigenvalues of $\mathcal{A}$. It was proved that $\lambda$ is an eigenvalue
of $\mathcal{A}$ if and only if it is a root of the characteristic polynomial of
$\mathcal{A}$ (see details in \cite{Shao:Connected odd-bipartite}).

The following result for nonnegative tensors is stated as a part of Perron-Frobenius theorem in \cite{K.C.Chang.etc:Perron-Frobenius Theorem}.
\begin{theorem}[\cite{K.C.Chang.etc:Perron-Frobenius Theorem}]
\label{thm:Perron-Frobenius}
Let $\mathcal{A}$ be a nonnegative tensor of order $r$ and dimension $n$. Then we have
the following statements.
\begin{enumerate}
\item $\rho(\mathcal{A})$ is an eigenvalue of $\mathcal{A}$ with a nonnegative
eigenvector corresponding to it.

\item If $\mathcal{A}$ is weakly irreducible, then $\rho(\mathcal{A})$ is the
unique eigenvalue of $\mathcal{A}$ with the unique eigenvector $\bm{x}\in\mathbb{R}_{++}^n$,
up to a positive scaling coefficient.
\end{enumerate}
\end{theorem}

\begin{theorem}
[\cite{Qi2013}]
\label{relaigh}
Let $\mathcal{A}$ be a nonnegative symmetric tensor of order $r$ and dimension $n$.
Then we have
\[
\rho(\mathcal{A})=\max\left\{\bm{x}^{\mathrm{T}}(\mathcal{A}\bm{x})\,|\, x\in\mathbb{R}_{+}^{n},
||\bm{x}||_r=1\right\}.
\]
Furthermore, $\bm{x}\in\mathbb{R}_{+}^{n}$ with $||\bm{x}||_r=1$ is an optimal
solution of the above optimization problem if and only if it is an eigenvector of
$\mathcal{A}$ corresponding to the eigenvalue $\rho(\mathcal{A})$.
\end{theorem}

In 2012, Cooper and Dutle \cite{cooper2012spectra} defined the
adjacency tensors for $r$-uniform hypergraphs.

\begin{definition}[\cite{cooper2012spectra}]
Let $\mathcal{H}=(V, E)$ be an $r$-uniform hypergraph on $n$ vertices. The adjacency tensor of $\mathcal{H}$ is defined as the order $r$ and dimension $n$ tensor $\mathcal{A}(\mathcal{H})=\left(a_{i_1 i_2 \ldots i_r}\right)$, whose ($i_1 i_2 \ldots i_r$)-entry is
$$
a_{i_1 i_2 \ldots i_r}= \left\{\begin{array}{cl}
\frac{1}{(r-1)!}, & \text { if }\left\{v_{i_1}, v_{i_2}, \ldots, v_{i_r}\right\} \in E(\mathcal{H}), \\[8pt]
0, & \text { otherwise}.
\end{array}\right.
$$
\end{definition}

Let $\mathcal{D}(\mathcal{H})$ be the \emph{diagonal degree tensor} of order $r$ and dimension $n$, where the diagonal entries are given by $d_{ii\ldots i} = \sum_{i_2, \ldots, i_r=1}^n a_{i i_2 \ldots i_r}$ for all $i \in [n]$, and all other entries are zero.
Consequently, the \emph{signless Laplacian tensor} of $\mathcal{H}$ is defined as $\mathcal{Q}(\mathcal{H})=\mathcal{D}(\mathcal{H})+\mathcal{A}(\mathcal{H})$. Moreover, we shall write $\lambda(\mathcal{H})$ and $q(\mathcal{H})$ for the spectral radius of the adjacency tensor and the signless Laplacian tensor of an $r$-graph $\mathcal{H}$, respectively. Let $\mathfrak{H}$ be a family of $r$-graphs. We define $q(\mathfrak{H})=\max\{q(\mathcal{H})\ |\ \mathcal{H}\in\mathfrak{H}\}$.
For the case $r=2$, these definitions reduce to the classical definitions for graphs.

In \cite{Friedland}, the weak irreducibility of nonnegative tensors was defined. It
was proved that an $r$-uniform hypergraph $H$ is connected if and only if its adjacency
tensor $\mathcal{A}(H)$ is weakly irreducible (see \cite{Friedland}). Clearly, this shows that if $H$
is connected, then $\mathcal{A}(H)$ and $\mathcal{Q}(H)$ are all
weakly irreducible.

We now list some classical results on spectral Tur\'an problems of graphs.
In 1986, Wilf \cite{wilf1986spectral} demonstrated that if $G$ is an $n$-vertex $K_{k+1}$-free graph, then $\lambda(G)\le(1-1/k)n$.
Motivated by this, Nikiforov~\cite{nikiforov2007bounds} and Guiduli~\cite{guiduli1996spectral} independently proved that such a graph $G$ must satisfy $\lambda(G) \leq \lambda\left(T_{n, k}\right)$, thereby extending Tur\'an's Theorem to the spectral version.
In 2013, Abreu and Nikiforov~\cite{deabreu2013maxima} proved that if $G$ is an $n$-vertex $K_{k+1}$-free graph, then $q(G) \leq\left(1-1/k\right) 2 n$.
Furthermore, He, Jin, and Zhang~\cite{he2013sharp} showed that such a graph $G$ satisfies $q(G) \leq q\left(T_{n, k}\right)$, thereby obtaining a spectral version of Tur\'an's theorem in terms of the signless Laplacian spectral radius. We say that a graph $F$ is color-critical if there exists an edge $e$ of $F$ such that $\chi(F-e)<\chi(F)$, where $F-e$ denotes the graph obtained from $F$ by deleting the edge $e$.
Let  $F$  be a color-critical graph with   $\chi(F)=r+1$. For sufficiently large $n$, Simonovits \cite{simonovits1968method} determined the Tur\'an number of $F$, proving that the unique extremal graph on $n$ vertices is $T_{n,r}$. Nikiforov \cite{nikiforov2009spectral} established a spectral counterpart to this result, showing that $T_{n,r}$ attains the maximum adjacency spectral radius among all $F$-free graphs on $n$ vertices. Recently, Zheng, Li, and Li \cite{zheng2026signless} settled the signless Laplacian spectral extremal problem for $F$, determining the maximum signless Laplacian spectral radius explicitly. For more related information, we refer the reader to \cite{he2013sharp, li2022survey, nikiforov2011some, wang2023conjecture, zheng2025some}.

Regarding the spectral Tur\'an problem of hypergraphs, Keevash, Lenz, and Mubayi~\cite{keevash2014spectral}  introduced two general criteria for the $\alpha$-spectral radius of $r$-uniform hypergraphs (which reduces to the adjacency spectral radius when $\alpha=r$)
that can be applied to obtain a variety of spectral Tur\'an-type results. In particular,
they determined the maximum $\alpha$-spectral radius of any $3$-graph on $n$ vertices not
containing the Fano plane when $n$ is sufficiently large.  Extending the spectral Mantel's theorem to hypergraphs,
Ni, Liu and Kang \cite{Ni-Liu-Kang2024} determined the maximum $\alpha$-spectral radius
of $\{F_4, F_5\}$-free $3$-graphs, and characterized the extremal hypergraph,
where $F_4 = \{abc, abd, bcd\}$ and $F_5 = \{abc, abd, cde\}$. Recently, Zheng,
Li, and Fan \cite{Zheng-Li-Fan2024} established the maximum $\alpha $-spectral radius
of $K_{k+1}^{(r)}$-free $r$-graphs by leveraging results from Keevash, Lenz, and
Mubayi \cite{keevash2014spectral}, and further characterized the extremal
hypergraph attaining this bound. Hou, Liu, and Zhao \cite{hou2024criterion} established a simple criterion for degree-stability of hypergraphs. An immediate application of this result, combined with the general theorem by Keevash, Lenz, and Mubayi~\cite{keevash2014spectral}, solves the $\alpha$-spectral Tur\'an problems for a large class of hypergraphs.

Research has also been conducted on spectral Tur\'an-type problems over specific
classes of hypergraphs.
Gao, Chang and Hou \cite{GaoChangHou2022}
investigated the spectral extremal problem for $K_{r+1}^{(r)}$-free $r$-graphs among linear
hypergraphs. They proved that the spectral radius of an $n$-vertices
$K_{r+1}^{(r)}$-free linear $r$-graph is no more than $n/r$ when $n$ is sufficiently large.
Generalizing Gao, Chang and Hou's result, She, Fan, Kang and Hou \cite{She-Fan-Kang-Hou2023}
presented sharp (or asymptotic) bounds of the spectral radius of $F^{(r)}$-free linear
$r$-graphs by establishing the connection between the spectral radius of linear
hypergraphs and those of their shadow graphs.  Here, $F$ is a graph  with chromatic
number $k \geq r +1$.
Another relevant result,
due to Ellingham, Lu and Wang \cite{Ellingham-Lu-Wang2022}, showed that the
$n$-vertex outerplanar $3$-graph of maximum spectral radius is the unique $3$-graph
whose shadow graph is the join of an isolated vertex and the path $P_{n-1}$.

Despite these advances,   the signless Laplacian spectral version has received much less attention, partly because of the additional structural information encoded in the degree tensor.  In this paper, we establish a criterion for the signless Laplacian spectral Tur\'an problem of hypergraphs, as stated in Theorem~\ref{criterion}. Furthermore, as an application of Theorem~\ref{criterion}, we determine the maximum signless Laplacian spectral radius among all Fano plane-free $3$-graphs on $n$ vertices.

\section{Main results}
We establish a general criterion which, combined with further arguments, enables us to derive exact signless Laplacian spectral Tur\'an results for certain hypergraphs.
\begin{theorem}\label{criterion}
Let $\mathcal{F}$ be a family of $r$-graphs with $\pi(\mathcal{F}) > 1/2$. Let $\varepsilon$ and $\sigma$ be sufficiently small numbers such that $\varepsilon \gg \sigma > 0$. Denote by $\mathcal{H}_n$ the set of all $\mathcal{F}$-free $r$-graphs on $n$ vertices with minimum degree greater than $\bigl( \pi(\mathcal{F})/(r-1)! - \varepsilon \bigr) n^{r-1}$. Suppose that there exists $N > 0$ such that for every $n > N$, we have
\begin{equation}
  \biggl| \mathrm{ex}_r(n, \mathcal{F}) - \mathrm{ex}_r(n-1, \mathcal{F}) - \frac{\pi(\mathcal{F})}{(r-1)!} n^{r-1} \biggr| \leq \sigma n^{r-1} \label{condition 1}
\end{equation}
and
\begin{equation}
  \biggl| q(\mathcal{H}_n) - 2r \frac{\mathrm{ex}_r(n, \mathcal{F})}{n} \biggr| \leq \sigma n^{r-2}. \label{condition 2}
\end{equation}
Then there exists $n_0 > 0$ such that for any $\mathcal{F}$-free $r$-graph $\mathcal{G}$ on $n \geq n_0$ vertices, we have
\[
  q(\mathcal{G}) \leq q(\mathcal{H}_n).
\]
In addition, the equality holds only if $\mathcal{G} \in \mathcal{H}_n$.
\end{theorem}

The Galois field of order 2, denoted by $\mathrm{GF}(2)$, is the finite field consisting of the elements $\{0, 1\}$ equipped with addition and multiplication modulo 2. The Fano plane, denoted by $\mathrm{PG}_2(2)$ (see Figure~\ref{fig:fano}), is the projective plane over $\mathrm{GF}(2)$. In hypergraph  theory, the Fano plane, whose study was suggested by Vera T. S\'os, is identified with the hypergraph having vertex set $[7]$ and edge set $\{123, 345, 561, 174, 275, 376, 246\}$.
\begin{figure}[H]
    \centering
    \includegraphics[width=0.4\linewidth]{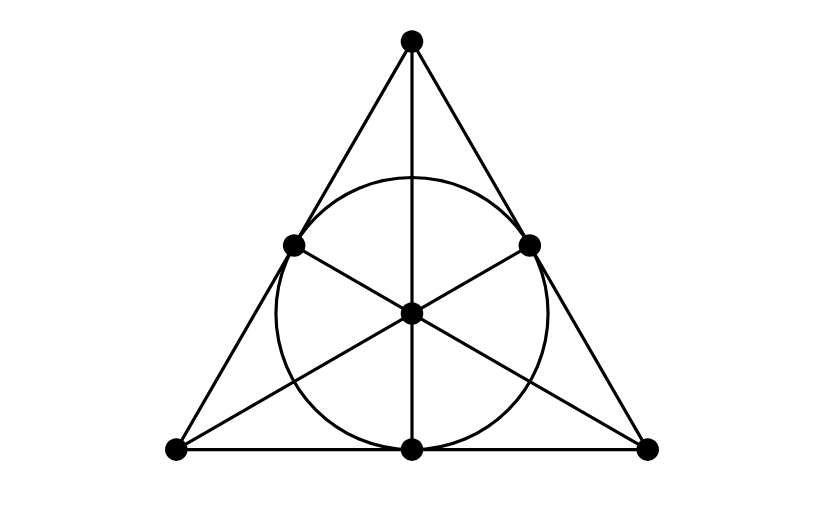}
    \caption{Fano plane}
    \label{fig:fano}
\end{figure}

Denote by $\mathcal{B}_n$ the balanced complete $2$-colorable $3$-uniform hypergraph on $n$ vertices. By applying Theorem \ref{criterion}, we determine the signless Laplacian spectral extremal hypergraph for the Fano plane.

\begin{theorem}\label{Fano}
Let $\mathcal{G}$ be a $\mathrm{PG}_2(2)$-free $3$-uniform hypergraph on $n$ vertices. Then for sufficiently large $n$, we have $q(\mathcal{G}) \leq q(\mathcal{B}_n)$, with equality if and only if $\mathcal{G}=\mathcal{B}_n$.
\end{theorem}

\section{Preliminaries}
Given a vector $\mathbf{x}=(x_1, \ldots, x_n) \in \mathbb{R}^n$, for a subset \scalebox{0.95}{$U \subseteq V(\mathcal{H})$}, let \scalebox{0.95}{$x^U=\prod_{v \in U} x_v$}. By
Theorem \ref{relaigh}, the signless Laplacian spectral radius $q(\mathcal{H})$ can be characterized as\\[-6mm]
\begin{align}\label{eq3}
    q(\mathcal{H})=\max _{\|\mathbf{x}\|_r=1}\Biggl(\sum_{e\in E(\mathcal{H})}\bigg(\sum_{v\in e} x_v^r+r  x^e\bigg)\Biggr).
\end{align}
It is well known that there always exists a non-negative eigenvector corresponding to $q(\mathcal{H})$. Furthermore, if $\mathcal{H}$ is connected, the unique positive unit eigenvector $\mathbf{x}$ corresponding to $q(\mathcal{H})$ is called the \emph{principal eigenvector} of $\mathcal{H}$. By the eigenequation, the following holds for each $v \in V(\mathcal{H})$:
\begin{align*}
    q(\mathcal{H}) x_v^{r-1} = d(v) x_v^{r-1} + \sum_{e \in E_v} x^{e \setminus \{v\}}.\\[-7.5mm]
\end{align*}


While the Tur\'an problem for graphs has been extensively studied, the problem becomes significantly more challenging for hypergraphs. In fact, very few exact results in hypergraphs are known. Even for the complete $r$-graph $K_s^{(r)}$ on $s$ vertices, the natural generalization of the complete graph $K_s$, the exact Tur\'an numbers remain unknown for most values. Since Tur\'an-type problems in hypergraphs are notoriously difficult, most progress in this area relies on the study of the stability phenomenon.

Many families $\mathcal{F}$ of $r$-graphs have the property that there exists a unique $\mathcal{F}$-free $r$-graph $\mathcal{G}$ on $n$ vertices with $e(\mathcal{G})=\operatorname{ex}_r(n, \mathcal{F})$, and every $\mathcal{F}$-free hypergraph of size close to $\operatorname{ex}_r(n, \mathcal{F})$ can be transformed to $\mathcal{G}$ by deleting and adding very few edges. Such a property is called stability of $\mathcal{F}$.
The first stability theorem was proved  by Erd\H{o}s and Simonovits~\cite{simonovits1968method}.

\begin{theorem}[\cite{simonovits1968method}]\label{stab.1}
  Let $\ell \geq 2$ and let $\mathcal{F}$ be a family of graphs with $\chi(\mathcal{F}) = \ell+1$.
  Then for every $\delta > 0$, there exist $\varepsilon > 0$ and $N_0 \in \mathbb{N}$ such that every $\mathcal{F}$-free graph on $n \geq N_0$ vertices with at least $(1-\varepsilon) \operatorname{ex}(n, \mathcal{F})$ edges can be transformed into the Tur\'an graph $T_{n,\ell}$ by deleting and adding at most $\delta n^2$ edges.
\end{theorem}

Using Theorem~\ref{stab.1}, Simonovits determined the exact value of $\operatorname{ex}(n, F)$ for all edge-critical graphs $F.$ Over the years, the notion of stability has evolved into diverse variants and yielded numerous results (e.g., see \cite{chen2024strong, erdos1988how, hou2024criterion, korandi2021exact, liu2023stability,liu2024hypergraph,liu2023unified}). A well-known variation involves replacing the number of edges with minimum degree conditions, a concept we simply refer to as \emph{degree-stability}.

Let $\mathcal{F}$ be a family of graphs and $\mathcal{G}$ be a class of $\mathcal{F}$-free graphs. If there exist $\varepsilon>0$ and $N>0$ such that every $\mathcal{F}$-free graph $G$ on $n\geq N$ vertices with $\delta(G)\geq (\pi(\mathcal{F})-\varepsilon)n$ is a subgraph of some member of $\mathcal{G}$, then we say that $\mathcal{F}$ is \textit{degree-stable} with respect to $\mathcal{G}$.

Such stability arguments have also been widely employed to investigate Tur\'an-type problems.
Regarding the Tur\'an problem, a notable early application of the stability method to hypergraphs concerns the Fano plane, established by Keevash and Sudakov~\cite{keevash2005turan}, and independently by F\"uredi and Simonovits~\cite{furedi2005triple}.  Below, we summarize the known results on the Tur\'an number of the Fano plane and its degree-stability.

De Caen and F\"uredi~\cite{decaen2000maximum}  determined that $\pi(\mathrm{PG}_2(2))=3/4$. The exact Tur\'an number of the Fano plane is then given by the following theorem.
\begin{theorem}[\cite{furedi2005triple,keevash2005turan}]\label{fano turan number}
  There exists an $n_1$ such that the following holds. If $\mathcal{H}$ is a triple system on $n > n_1$ vertices not containing the Fano configuration $\mathrm{PG}_2(2)$ and of maximum cardinality, then it is $2$-colorable. Thus,
  \[
    \operatorname{ex}_3(n, \mathrm{PG_2(2)}) = \binom{n}{3} - \binom{\lfloor n/2 \rfloor}{3} - \binom{\lceil n/2 \rceil}{3}.
  \]
\end{theorem}

Given a $3$-graph $\mathcal{H}$ and a partition $V(\mathcal{H}) = X \cup \bar{X}$, let $\mathcal{H}(X, \bar{X})$ denote the subgraph formed by all edges meeting both $X$ and $\bar{X}$. The following theorem establishes the degree-stability of the Fano plane.
\begin{theorem}[\cite{furedi2005triple, keevash2005turan}]\label{Fano-deg-stability}
  There exist a $\gamma_2 > 0$ and an $n_2$ such that the following holds. If $\mathcal{H}$ is a triple system on $n > n_2$ vertices not containing the Fano configuration $\mathrm{PG_2(2)}$ and
  \[
    d(v) > \left( \frac{3}{4} - \gamma_2 \right) \binom{n}{2}
  \]
  holds for every $v \in V(\mathcal{H})$, then $\mathcal{H}$ is bipartite, that is, $\mathcal{H} \subseteq \mathcal{H}\left(X, \bar{X}\right)$ for some $X \subseteq V(\mathcal{H})$.
\end{theorem}

We will use the following simple facts.
\renewcommand{\thefact}{\arabic{fact}}

\setcounter{fact}{0} \begin{fact}(H\"older's inequality).\label{Holder}
 Suppose that $(x_1, \ldots, x_n)$ and $(y_1, \ldots, y_n)$ are nonnegative vectors and $p, q$ are positive integers satisfying $\frac{1}{p}+\frac{1}{q}=1$. Then
 $$\sum_{i\in [n]}x_iy_i\leq (x_1^p+\cdots+x_n^p)^{\frac{1}{p}}(y_1^q+\cdots+y_n^q)^{\frac{1}{q}}.$$
 In particular, taking $(x_1, \ldots, x_n)=(1, \ldots, 1)$ and $(p,q)=(\frac{r}{r-1}, r)$, we obtain
 $$\sum_{i\in [n]}y_i\leq n^{\frac{r-1}{r}}(y_1^r+\cdots+y_n^r)^{\frac{1}{r}}.$$

\end{fact}
\setcounter{fact}{1} \begin{fact}\label{G-B-T}
  Let $\alpha \ge 1$ be an integer. If $x > 0$ is sufficiently small, then
  \[
    (1+x)^\alpha < 1 + \bigl(\alpha + \tfrac12\bigr)x.
  \]
\end{fact}

\begin{proof}
  By the binomial expansion,
  \[
    (1+x)^\alpha = \sum_{k=0}^\alpha \binom{\alpha}{k} x^k
    = 1 + \alpha x + \sum_{k=2}^{\alpha} \binom{\alpha}{k} x^k.
  \]
Clearly, there exists a constant $C$ such that for every sufficiently small $x>0$, we have
  \[
    \sum_{k=2}^{\alpha} \binom{\alpha}{k} x^k \le C(\alpha-2)x^2\le \frac{x}{2}.
  \]
  Thus,
  \[
    (1+x)^\alpha
    = 1 + \alpha x + \sum_{k=2}^{\alpha} \binom{\alpha}{k} x^k
    < 1 + \biggl(\alpha + \frac{1}{2}\biggr)x.
  \]
  This completes the proof.\qedhere
\end{proof}

The following  facts,  proved by Zheng, Li, and Li \cite{zheng2026signless}, will be employed in our proof.
\begin{fact}[\cite{zheng2026signless}]\label{fact1}
    If $0<x<\frac{1}{2}$ and $0<a<1$, then $\ln (1-a x)+a x+x^2>0$.
\end{fact}

\begin{fact}[\cite{zheng2026signless}]\label{fact2}
    If $x>1$, then $\frac{1}{x}<\ln x-\ln (x-1)$ and $\frac{1}{x^2}<\frac{1}{x-1}-\frac{1}{x}$.
\end{fact}
\begin{fact}(Young's inequality). \label{fact6}
For non-negative real number $a, b\geq 0$, real number $p, q>1$ satisfying $\frac{1}{p}+\frac{1}{q}=1$, then
$$ab\leq \frac{a^p}{p}+\frac{b^q}{q}$$
with equality if and only if $a^p=b^q$.

\end{fact}

\begin{fact}\label{yusuan}
  Let $a, b, u, v, f, g$ be positive real numbers. If $a u + b v = 1$, then
  \[
    f u + g v \leq \max \left\{ \frac{f}{a}, \frac{g}{b} \right\}.
  \]
\end{fact}
\begin{proof}
  Note that $u$ and $v$ are positive. Then
\[
  f u = a u\cdot\frac{f}{a}\leq a u \cdot \max \biggl\{ \frac{f}{a}, \frac{g}{b} \biggr\}
\]
and
\[
  g v =b v\cdot\frac{g}{b}\leq b v \cdot \max \biggl\{ \frac{f}{a}, \frac{g}{b} \biggr\}.
\]
Summing these inequalities and using $a u + b v = 1$, we obtain
  \[
    f u + g v \leq (a u + b v) \max \left\{ \frac{f}{a}, \frac{g}{b} \right\} = \max \left\{ \frac{f}{a}, \frac{g}{b} \right\}. \qedhere
  \]
\end{proof}

\section{Proof of Theorem \ref{criterion}.}
 \noindent
 Suppose that $\mathcal{F}, \sigma, n$ and $\mathcal{H}_n$ satisfy \eqref{condition 1} and \eqref{condition 2} for all integers  $n \geq N$.
 A standard argument shows that the quotient \scalebox{0.95}{$\mathrm{ex}_r(n, \mathcal{F}) /\binom{n}{r}$} is non-increasing and converges to $\pi(\mathcal{F}).$
 Moreover, there exist $0 < \tau \ll \varepsilon$ and $N_1$ satisfying \scalebox{0.95}{$\pi(\mathcal{F}) \leq \mathrm{ex}_r(n, \mathcal{F}) \big/ \binom{n}{r}  < \pi(\mathcal{F})+\tau$ for all $n \geq N_1$}.
Combining with \eqref{condition 2}, we derive the bounds
$$
2 r \pi(\mathcal{F}) \frac{\binom{n}{r}}{n}-\sigma n^{r-2} \le q\left(\mathcal{H}_n\right) \le 2 r(\pi(\mathcal{F})+\tau) \frac{\binom{n}{r}}{n}+\sigma n^{r-2}.
$$
Accordingly, there exists a constant  $\gamma$ satisfying $\tau<\gamma\ll \sigma$  such that the following statement holds:
\begin{align}
 \Bigg(2 \frac{\pi(\mathcal{F})}{(r-1)!}-\gamma \Bigg) n^{r-1}\leq   q\left(\mathcal{H}_n\right)\leq\Bigg(2 \frac{\pi(\mathcal{F})}{(r-1)!}+\gamma \Bigg) n^{r-1}.\label{condition 3}
\end{align}
 Using \eqref{condition 1} and \eqref{condition 2}, we obtain
\begin{align}
   & \biggl| q(\mathcal{H}_n) - q(\mathcal{H}_{n-1}) - \frac{2\pi(\mathcal{F})}{(r-2)!} n^{r-2} \biggr| \notag \\[2mm]
  \leq \, & \biggl| 2 r\frac{ \mathrm{ex}_r(n, \mathcal{F})}{n} -2 r  \frac{\mathrm{ex}_r(n-1, \mathcal{F})}{n-1} - \frac{2\pi(\mathcal{F})}{(r-2)!} n^{r-2} \biggr| + 2 \sigma n^{r-2} \notag \\[2mm]
  = \, & \biggl| \frac{2 r}{n} \Bigl( \mathrm{ex}_r(n, \mathcal{F}) - \mathrm{ex}_r(n-1, \mathcal{F}) - \frac{\pi(\mathcal{F})}{(r-1)!} n^{r-1} \Bigr) + \frac{2 \pi(\mathcal{F})}{(r-1)!} n^{r-2} - \frac{2 r \mathrm{ex}_r(n-1, \mathcal{F})}{n(n-1)} \biggr|\notag\\[2mm]
  & + 2 \sigma n^{r-2} \notag \\[2mm]
  \leq \, & 2\biggl| \frac{ \pi(\mathcal{F})}{(r-1)!} n^{r-2} -  \frac{r \mathrm{ex}_r(n-1, \mathcal{F})}{n(n-1)} \biggr| + (2 r+2) \sigma n^{r-2}\notag \\[2mm]
  \leq \, & (2 r+4) \sigma n^{r-2}. \label{condition 4}
\end{align}

\medskip

\noindent\textbf{Proof of Theorem \ref{criterion}}.  Let $\mathcal{F}$ and $\mathcal{H}_n$ be the  families of $r$-graphs described in Theorem \ref{criterion}. Let $\mathcal{G}$ be an $n$-vertex $\mathcal{F}$-free $r$-graph, and let  $\mathbf{x}=(x_1, \dots, x_n)$ be a non-negative unit eigenvector corresponding to $q(\mathcal{G})$. In the following, we will show that if $q(\mathcal{G}) \geq q(\mathcal{H}_n)$, then $\mathcal{G} \in \mathcal{H}_n$. We prove by contradiction. If $\mathcal{G} \notin \mathcal{H}_n$, then we have the following inequalities immediately:
\begin{gather}
   \delta(\mathcal{G}) \leq \Bigl( \frac{\pi(\mathcal{F})}{(r-1)!} - \varepsilon \Bigr) n^{r-1}, \label{hypo.condition.1} \\
   q(\mathcal{G}) \geq q(\mathcal{H}_n) \geq \Bigl( 2 \frac{\pi(\mathcal{F})}{(r-1)!} - \gamma \Bigr) n^{r-1}. \label{hypo.condition.2}
\end{gather}
Using the above conditions, we establish a critical property of the entries of the eigenvector $\mathbf{x}$ corresponding to $q(\mathcal{G})$. Let  $w\in V(\mathcal{\mathcal{\mathcal{G}}})$ with $x_{w}=\min \{x_{1}, \dots, x_{n}\}$.

\begin{lemma} \label{min}
 For sufficiently large $n$, we have
$$x_w^r<\frac{1-\varepsilon}{n}.$$
\end{lemma}
\begin{proof}
Suppose to the contrary that $x_w^r\geq ({1-\varepsilon})/{n}$. Then $x_v^r\geq(1-\varepsilon)/{n} $ for every $v\in V(\mathcal{G})$. Let $u\in V(\mathcal{G})$ be a vertex with $d(u)=\delta(\mathcal{G})$. From the eigenequation for $q(\mathcal{G})$ at the vertex $u$, we obtain
$$
\big(q(\mathcal{G)}-\delta(\mathcal{G})\big) x_u^{r-1}=\sum_{e \in E_u} x^{e\setminus\{u\}}.
$$
By combining \eqref{hypo.condition.1} and \eqref{hypo.condition.2}, we derive a lower bound for the sum $\sum_{e \in E_u} x^{e\setminus\{u\}}$ as follows:
\begin{align}
\sum_{e \in E_u} x^{e\setminus\{u\}}
&= \big(q(\mathcal{G})-\delta(\mathcal{G})\big) x_u^{r-1}\geq  \Big(\frac{\pi(\mathcal{F})}{(r-1)!} +\varepsilon-\gamma\Big) \left(1-\varepsilon\right)^\frac{r-1}{r}n^{r-2+\frac{1}{r}}.\label{lower bound}
\end{align}
On the other hand, we can establish an upper bound for $\sum_{e \in E_u} x^{e\setminus\{u\}}.$
Firstly, it follows from Fact~\ref{Holder}
 that
\begin{align}
    \sum_{e \in E_u} x^{e\setminus\{u\}}\leq {d(u)}^\frac{r-1}{r}\Bigg(\sum_{e \in E_u} \big(x^{e\setminus\{u\}}\big)^r\Bigg)^{\frac{1}{r}}.\label{upper 1}
\end{align}
Next, we  estimate the upper bound for \scalebox{0.95}{$\sum_{e \in E_u} \big(x^{e\setminus\{u\}}\big)^r$}. One can verify that
\begin{align}
 \sum_{e \in E_u} \Bigl(x^{e\setminus\{u\}}\Bigr)^r
 &= \sum_{S \in E_u^{r-1}}\bigl(x^S\bigr)^r- \! \sum_{S \in E_u^{r-1} \setminus L(u)}\bigl(x^S\bigr)^r \notag \\[0mm]
  &= \frac{1}{(r-1)!} \Bigg( \sum_{\bm{i}\in V_u^{r-1}} \bigl(x^{\bm{i}}\bigr)^r - \!\! \sum_{\bm{i}\in V_u^{r-1} \setminus L^o(u)} \bigl(x^{\bm{i}}\bigr)^r \Bigg) \notag \\[1mm]
  &= \frac{1}{(r-1)!} \Bigg(\prod_{j=1}^{r-1} \sum_{i_j \in V(\mathcal{G}-u)} x_{i_j}^r - \!\! \sum_{\bm{i} \in V_u^{r-1} \setminus L^o(u)}\bigl(x^{\bm{i}}\bigr)^r\Bigg) \notag \\[2mm]
  &\leq \frac{1}{(r-1)!} \left( 1 - \big( (n-1)^{r-1} - ( \pi(\mathcal{F}) - (r-1)! \varepsilon ) n^{r-1} \big) \cdot \left( \frac{1-\varepsilon}{n} \right)^{r-1} \right) \notag \\[2mm]
  &\leq \frac{1}{(r-1)!} \big( 1 - ( 1 - \pi(\mathcal{F}) ) (1-\varepsilon)^{r-1} \big). \label{upper 2}
\end{align}
Combining \eqref{upper 1} and \eqref{upper 2}, we conclude that
\begin{align*}
  \sum_{e \in E_u} x^{e \setminus \{u\}} \leq \biggl( \frac{\pi(\mathcal{F})}{(r-1)!} \biggr)^{\frac{r-1}{r}} \, \biggl( \frac{1}{(r-1)!} \Bigl( 1 - \bigl( 1 - \pi(\mathcal{F}) \bigr) (1-\varepsilon)^{r-1} \Bigr) \biggr)^{\frac{1}{r}} \, n^{r-2+\frac{1}{r}}.
\end{align*}
Together with \eqref{lower bound}, we deduce that
\begin{eqnarray}\label{eq10}
 & & \bigg( \frac{\pi(\mathcal{F})}{(r-1)!} + \varepsilon - \gamma \bigg) (1-\varepsilon)^{\frac{r-1}{r}}\nonumber\\
 & \leq &
  \bigg( \frac{\pi(\mathcal{F})}{(r-1)!} \bigg)^{\frac{r-1}{r}}
  \bigg( \frac{1}{(r-1)!} \Big( 1 - \big( 1 - \pi(\mathcal{F}) \big) (1-\varepsilon)^{r-1} \Big) \bigg)^{\frac{1}{r}}.
\end{eqnarray}
Let
  \[
    g = \frac{(r-1)! \left( \frac{\pi(\mathcal{F})}{(r-1)!} + \varepsilon - \gamma \right) (1-\varepsilon)^{(r-1)/r}}{\pi(\mathcal{F})^{(r-1)/r} \left( 1 - (1-\pi(\mathcal{F}))(1-\varepsilon)^{r-1} \right)^{1/r}}.
  \]
 We will show that   $g>1$, which contradicts \eqref{eq10} and thus completes the proof.
For convenience, we abbreviate  $\pi(\mathcal{F})$  to $\pi$. We simplify  $g$  as follows:
  \begin{align*}
    g &= \frac{\Big( \pi + (r-1)!(\varepsilon-\gamma) \Big) (1-\varepsilon)^{(r-1)/r}}{\pi^{(r-1)/r} \Big( 1 - (1-\pi)(1-\varepsilon)^{r-1} \Big)^{1/r}} \\[2mm]
      &= \left( \frac{\pi \Big(1 + (r-1)!(\varepsilon-\gamma) / \pi\Big)^r (1-\varepsilon)^{r-1}}{1 - (1-\pi)(1-\varepsilon)^{r-1}} \right)^{1/r}.
  \end{align*}
By Bernoulli's inequality, we obtain
$$ g \ge \left( \frac{\pi + r!(\varepsilon-\gamma)}{\left( \frac{1}{1-\varepsilon}\right)^{r-1} - (1-\pi)} \right)^{1/r}.$$
Now, it suffices to show that
  \[
    T(\varepsilon,\gamma,r) = 1 + r!(\varepsilon-\gamma) - \left( \frac{1}{1-\varepsilon} \right)^{r-1} > 0.
  \]
Using Fact~\ref{G-B-T}, we obtain
  \begin{align*}
         T(\varepsilon,\gamma,r)&= 1 + r!(\varepsilon-\gamma)-\left(1+\frac{\varepsilon}{1-\varepsilon}\right)^{r-1} \\[2mm]
    &\ge r!(\varepsilon-\gamma) - \frac{(r-1/2)\varepsilon}{1-\varepsilon}\\[2mm]
    & >\varepsilon \left( r! - \frac{r-1/2}{1-\varepsilon} \right) - r!\gamma.
  \end{align*}
 Since $\varepsilon$ is sufficiently small, we have $r! - \frac{r-1/2}{1-\varepsilon} > 0$. Furthermore, the term $\varepsilon\Bigl(r! - \frac{r-1/2}{1-\varepsilon}\Bigr)$ dominates $r!\gamma$  as $\gamma \ll \varepsilon$, implying that $T(\varepsilon,\gamma,r) > 0$.
 The proof is complete.
\end{proof}

\medskip
\begin{lemma}\label{calcu}
 $$q(\mathcal{G}-w) \geq \frac{1-r x_w^r}{1-x_w^r} q(\mathcal{G})-\frac{n^{r-2}}{(r-2)!}\frac{1-(n-1)x_w^r}{1-x_w^r}.$$
\end{lemma}
\begin{proof}
Firstly, by Theorem \ref{relaigh} and (\ref{eq3}), we derive that
\begin{align}
  q(\mathcal{G})  &= \sum_{e \in E(\mathcal{G})}  \biggl(\sum_{v \in e}x_v^r + r x^e \biggr) \notag \\[1mm]
  &= \sum_{e \in E(\mathcal{G}-w)}\biggl( \sum_{v \in e} x_v^r + r x^e \biggr) + \sum_{e \in E_w} \biggl( x_w^r + \sum_{v \in e \setminus \{w\}} x_v^r + r x_w \prod_{v \in e \setminus \{w\}} x_v \biggr) \notag \\[2mm]
  &\leq \bigl( 1 - x_w^r \bigr) q(\mathcal{G}-w) + d(w) x_w^r + \sum_{e \in E_w} \sum_{v \in e \setminus \{w\}} x_v^r + r x_w \sum_{e \in E_w} \prod_{v \in e \setminus \{w\}} x_v. \label{calcu 1}
\end{align}From the eigenequation for $q(\mathcal{\mathcal{G}})$ at the vertex $w$, we obtain
\[
  \bigl( q(\mathcal{G}) - d(w) \bigr) x_w^{r-1} = \sum_{e \in E_w} \prod_{v \in e \setminus \{w\}} x_v.
\]
Combining this identity with  \eqref{calcu 1}, we deduce that
\begin{align}
    q(\mathcal{G}-w) \geq\frac{\left(1-r x_w^r\right)}{1-x_w^r} q(\mathcal{G})-\frac{(1-r)d(w) x_w^r+\sum_{e \in E_w} \sum_{v \in e \setminus\{w\}} x_v^r}{1-x_w^r}.\label{calcu.2}
\end{align}
Next,  it follows from a direct calculation that
\begin{align}
\displaystyle\sum_{e \in E_w} \displaystyle\sum_{v \in e \setminus\{w\}} x_v^r
& =\displaystyle\sum_{S \in E_w^{r-1}} \displaystyle\sum_{v \in S} x_v^r-\displaystyle\sum_{S \in E_w^{r-1} \setminus L(w)} \displaystyle\sum_{v \in S} x_v^r \notag\\[2mm]
& =\frac{1}{(r-1)!}\left(\displaystyle\sum_{\left(i_1, \cdots, i_{r-1}\right) \in V_w^{r-1}}\displaystyle\sum_{j=1}^{r-1} x_{i_j}^r-\hspace{-4mm}\displaystyle\sum_{\left(i_1, \cdots, i_{r-1}\right) \in {V_w^{r-1}\setminus L^o(w)} }\displaystyle\sum_{j=1}^{r-1} x_{i_j}^r\right) \notag \\[2mm]
& \leq \frac{1}{(r-1)!}\Big((r-1) (n-1)^{r-2}-\left((n-1)^{r-1}-(r-1)!d(w)\right)(r-1) x_w^r\Big) \notag \\[2mm]
&=\frac{1}{(r-2)!}\Big((n-1)^{r-2}-\left((n-1)^{r-1}-(r-1)!d(w)\right) x_w^r\Big).\notag
\end{align}
Substituting the above inequality into (\ref{calcu.2}) gives the desired result
$$q(\mathcal{G}-w) \geq \frac{1-r x_w^r}{1-x_w^r} q(\mathcal{G})-\frac{n^{r-2}}{(r-2)!}\frac{1-(n-1)x_w^r}{1-x_w^r}.$$
This completes the proof.
\end{proof}

\medskip

\begin{lemma}\label{use calcu}
 For sufficiently large $n$, we have
\begin{align}
q(\mathcal{G}-w)\geq q(\mathcal{G})\left(1-\frac{r-1-\varepsilon^2/2}{n-1}\right),\label{use calcu.1}
\end{align}
and
\begin{align}
q(\mathcal{G}-w)> q(\mathcal{H}_{n-1})\label{use calcu.2}.
\end{align}
\end{lemma}
\begin{proof}
    We first prove the inequality (\ref{use calcu.1}). By Lemma \ref{calcu}, we deduce that
 \begin{align}
\frac{q(\mathcal{G}-w)}{n-r} \geq
& \frac{q(\mathcal{G})}{n-1}\left(1+\frac{r-1}{n-r}\right) \frac{1-r x_w^r}{1-x_w^r}-\frac{n^{r-2}}{(r-2)!} \frac{1-(n-1) x_w^r}{(n-r)\left(1-x_w^r\right)} \notag\\[2mm]
= & \frac{q(\mathcal{G})}{n-1}\left(1+\frac{(r-1)\left(1-n x_w^r\right)}{(n-r)\left(1-x_w^r\right)}\right)-\frac{n^{r-2}}{(r-2)!} \frac{1-(n-1) x_w^r}{(n-r)\left(1-x_w^r\right)} \notag
\end{align}
For convenience, we denote $t:=t(n,r,x_w)=\frac{1-n x_w^r}{(n-r)\left(1-x_w^r\right)}>0.$ Therefore, we have
\begin{align}
  \frac{q(\mathcal{G}-w)}{n-r}
  \geq &\frac{q(\mathcal{G})}{n-1} \bigl( 1+(r-1)t \bigr) - \frac{n^{r-2}}{(r-2)!} t-2n^{r-4} \label{note} \\[2mm]
  =& \frac{q(\mathcal{G})}{n-1} \biggl( 1 + \Bigl( 1 - \frac{1}{2 \pi(\mathcal{F})} - \varepsilon \Bigr) (r-1)t \biggr)\notag \\[2mm]
 &+  \biggl( \Bigl( \frac{1}{2 \pi(\mathcal{F})} + \varepsilon \Bigr) \frac{(r-1)q(\mathcal{G})}{n-1} - \frac{n^{r-2}}{(r-2)!} \biggr) t -2n^{r-4}\notag \\[2mm]
  > &\frac{q(\mathcal{G})}{n-1} \biggl( 1 + \Bigl( 1 - \frac{1}{2 \pi(\mathcal{F})} - \varepsilon \Bigr) (r-1)t \biggr) \notag \\[2mm]
  >& \frac{q(\mathcal{G})}{n-1} \biggl( 1 + \frac{\varepsilon^2}{n} \biggr), \notag
\end{align}
where the second inequality follows from \eqref{hypo.condition.2} and the fact that $\varepsilon \gg \gamma$, and the last inequality follows from Lemma~\ref{min} and \scalebox{0.95}{$0 < \varepsilon < \frac{1}{2} \left( 1 - \frac{1}{2\pi(\mathcal{F})} \right)$}.
This implies that
\[
  q(\mathcal{G}-w) \geq q(\mathcal{G}) \Biggl( 1 - \frac{r-1}{n-1} \Biggr) \Biggl( 1 + \frac{\varepsilon^2}{n} \Biggr) \geq q(\mathcal{G}) \Biggl( 1 - \frac{r-1-\varepsilon^2/2}{n-1} \Biggr).
\]
Thus, the inequality \eqref{use calcu.1} holds.

Next, we proceed to prove inequality \eqref{use calcu.2}. By  \eqref{note}, we derive that
\begin{align}
  \frac{q(\mathcal{G}-w)}{n-r}
  &\geq \frac{q(\mathcal{G})}{n-1} \bigl( 1+(r-1) t \bigr) - \frac{n^{r-2}}{(r-2)!} t -2n^{r-4}\notag \\[2mm]
  &= \frac{q(\mathcal{G})}{n-1} + \biggl( \frac{(r-1)q(\mathcal{G})}{n-1} - \frac{n^{r-2}}{(r-2)!} \biggr) t -2n^{r-4}\notag \\[2mm]
  &\geq \frac{q(\mathcal{G})}{n-1} + \frac{\varepsilon}{2} t n^{r-2},\notag
\end{align}
where the last inequality follows from \eqref{hypo.condition.2} and the condition \scalebox{0.95}{$0 < \varepsilon < \frac{1}{2} \left( 1 - \frac{1}{2\pi(\mathcal{F})} \right)$}.
Thus, we deduce that
\begin{align*}
  q(\mathcal{G}-w) &\geq q(\mathcal{G}) \biggl( 1-\frac{r-1}{n-1} \biggr) + \frac{\varepsilon}{2} \frac{1-n x_w^r}{1-x_w^r} n^{r-2}\\
 &\ge  q(\mathcal{G}) \biggl( 1-\frac{r-1}{n-1} \biggr) + \frac{\varepsilon^2}{2} n^{r-2}.
\end{align*}
From \eqref{condition 4} and \eqref{hypo.condition.2}, we obtain that
\begin{align*}
  q(\mathcal{G})\ge q(\mathcal{H}_n)\ge  q(\mathcal{H}_{n-1}) + \frac{2\pi(\mathcal{F}) }{(r-2)!} n^{r-2} - (2 r+4) \sigma n^{r-2}.
\end{align*}
Combining the two inequalities above, we further obtain
\[  q(\mathcal{G}-w) \ge \Big( q(\mathcal{H}_{n-1}) + \frac{2 \pi(\mathcal{F}) }{(r-2)!}n^{r-2} - (2 r+4) \sigma n^{r-2}\Big) \biggl( 1-\frac{r-1}{n-1} \biggr)+\frac{\varepsilon^2}{2}  n^{r-2} .\]
Then
\begin{eqnarray*}
   & &\frac{q(\mathcal{G}-w)-q(\mathcal{H}_{n-1})}{n^{r-2}}\\[2mm]
   &\ge&   \left( \frac{2 \pi(\mathcal{F})}{(r-2)!} - (2 r+4) \sigma \right) \left( 1-\frac{r-1}{n-1} \right) - \left( \frac{2\pi(\mathcal{F})}{(r-2)!}+(r-1)\gamma \right) \left(1-\frac{1}{n}\right)^{r-2} + \frac{\varepsilon^2}{2}\\[2mm]
   &\ge &  \left( \frac{2\pi(\mathcal{F})}{(r-2)!} - (2 r+4) \sigma \right) \left( 1-\frac{r-1}{n-1} \right) - \left( \frac{2\pi(\mathcal{F})}{(r-2)!}+(r-1)\gamma \right) + \frac{\varepsilon^2}{2}\\[2mm]
   &=&  (2 r+4) \frac{r-n}{n-1}\sigma - \frac{2 \pi(\mathcal{F})}{(r-2)!} \frac{r-1}{n-1} - (r-1) \gamma + \frac{\varepsilon^2}{2}\\
   &\ge & \frac{\varepsilon^2}{2}-(2r+4)\biggl(\sigma+\frac{1}{n-1}+\gamma\biggr)>0,
   \end{eqnarray*}
where the last inequality holds since  $\varepsilon \gg \sigma \gg \gamma > 0 $ and  $n$ is sufficiently large.
So the inequality (\ref{use calcu.2}) holds.
\end{proof}

We are now ready to demonstrate that the assumption $\mathcal{G} \notin \mathcal{H}_n$   leads to a contradiction. Let $N$   be sufficiently large such that Lemmas~\ref{min} and~\ref{use calcu} are applicable. Define
\begin{align}
  n_0 = \Bigl(r! N^{r-1} \mathrm{e}^{2(r-1)^2}\Bigr)^{2/\varepsilon^2}.\label{n_0}
\end{align}
Clearly, $n_0 > N$. Let $\mathbf{x}^n=(x_1^{(n)},x_2^{(n)},\cdots,x_n^{(n)})$ be a non-negative unit eigenvector of $\mathcal{Q}(\mathcal{G})$ corresponding to $q(\mathcal{G})$.  Suppose $w^{(n)}\in V(\mathcal{G})$ such that $x_{w^{(n)}}=\min\left\{x_1^{(n)},x_2^{(n)},\cdots,x_n^{(n)}\right\}$. Recall that $q(\mathcal{G})\ge q(\mathcal{H}_n) \ge \bigl( 2 \pi(\mathcal{F})/(r-1)! - \gamma \bigr) n^{r-1}$ and $\delta(\mathcal{G})\le  \bigl( \pi(\mathcal{F})/{(r-1)!} - \varepsilon \bigr) n^{r-1}$. Then, by Lemma \ref{min}, $\left(x_{w^{(n)}}\right)^r<(1-\varepsilon)/n.$
For $n> n_0$, let $\mathcal{G}_n=\mathcal{G}$ and $\mathcal{G}_{n-1}=\mathcal{G}_n-w^{(n)}$.
Then, it follows from Lemma \ref{use calcu} that
\begin{align}
  q(\mathcal{G}_{n-1}) \geq q(\mathcal{G}_n) \biggl( 1 - \frac{r-1 - \varepsilon^2/2}{n-1} \biggr), \notag
\end{align}
and
\begin{align}
q(\mathcal{G}_{n-1})> q(\mathcal{H}_{n-1}).\notag
\end{align}
By $q(\mathcal{G}_{n-1})> q(\mathcal{H}_{n-1})$, we have $ \mathcal{G}_{n-1}\notin \mathcal{H}_{n-1}$. So $\delta(\mathcal{G}_{n-1})\le  \bigl( \pi(\mathcal{F})/{(r-1)!} - \varepsilon \bigr) (n-1)^{r-1}.$  Let $\mathbf{x}^{n-1}=(x_1^{(n-1)},x_2^{(n-1)},\cdots,x_{n-1}^{(n-1)})$ be a  non-negative unit eigenvector of $\mathcal{Q}(\mathcal{G}_{n-1})$ corresponding to $q(\mathcal{G}_{n-1}).$ Suppose $w^{(n-1)}\in V(\mathcal{G}_{n-1})$ such that $x_{w^{(n-1)}}=\min\left\{x_1^{(n-1)},x_2^{(n-1)},\cdots,x_{n-1}^{(n-1)}\right\}$. Denote $\mathcal{G}_{n-2}=\mathcal{G}_{n-1}-w^{(n-1)}.$ By continuing this operation, we can obtain a sequence of hypergraphs $\mathcal{G}_n,\mathcal{G}_{n-1},\cdots,\mathcal{G}_N$ such that for each $n\ge i\ge N+1$,
\begin{align}
  q(\mathcal{G}_{i-1}) \geq q(\mathcal{G}_i) \biggl( 1 - \frac{r-1 - \varepsilon^2/2}{i-1} \biggr), \notag
\end{align}
and
\begin{align}
q(\mathcal{G}_{i-1})> q(\mathcal{H}_{i-1}).\notag
\end{align}
Then, for $i=N+1$, we have
\begin{align}
  q(\mathcal{G}_N)
  & \geq q(\mathcal{G}_{N+1}) \bigl( 1 - (r-1-\varepsilon^2/2) N^{-1} \bigr) \notag \\[2mm]
  & \geq q(\mathcal{G}_n) \prod_{i=N+1}^n \bigl( 1 - (r-1-\varepsilon^2/2) (i-1)^{-1} \bigr) \notag \\[1mm]
  & \geq q(\mathcal{G}_n) \exp \biggl( - \sum_{i=N+1}^n \Bigl( (r-1-\varepsilon^2/2)(i-1)^{-1} + (r-1)^2 (i-1)^{-2} \Bigr) \biggr) \notag \\[1mm]
  & \geq q(\mathcal{G}_n) \exp \biggl( - (r-1-\varepsilon^2/2) \ln \frac{n-1}{N-1} - (r-1)^2 \pi^2/6 \biggr) \notag \\[2mm]
  & \geq \biggl( \frac{2\pi(\mathcal{F})}{(r-1)!} - \gamma \biggr) n^{r-1} \biggl( \frac{n-1}{N-1} \biggr)^{-(r-1-\varepsilon^2/2)} e^{-(r-1)^2 \pi^2/6} \notag \\[2mm]
  & \geq \biggl(  \frac{2\pi(\mathcal{F})}{(r-1)!} - \gamma \biggr) e^{-(r-1)^2 \pi^2/6} n_0^{\varepsilon^2/2}\ge 2N^{r-1}. \notag
\end{align}
where the third inequality follows from Fact~\ref{fact1}, the fourth follows from Fact~\ref{fact2} together with the identity $\sum_{k=1}^{\infty} k^{-2} = \pi^2 / 6$, and the last inequality holds by \eqref{n_0} and the assumption $\pi(\mathcal{F})>\frac{1}{2}$.
This leads to a contradiction, as
\[
    q(\mathcal{G}_N) = \sum_{e \in E(\mathcal{G}_N)} \bigg( \sum_{v \in e} y_v^r + r y^e \bigg)
    \le 2 \sum_{e \in E(\mathcal{G}_N)} \sum_{v \in e} y_v^r
    = 2 \sum_{v \in V(\mathcal{G}_N)} d(v) y_v^r
    < 2 N^{r-1},
\]
where $\mathbf{y}$ is a  non-negative unit eigenvector of $\mathcal{Q}(\mathcal{G}_{N})$ corresponding to $q(\mathcal{G}_{N}).$ Here, the first inequality follows from the AM-GM inequality. The proof is complete.\qed
\\

\section{Proof of Theorem \ref{Fano}}
In this section, we will prove Theorem \ref{Fano}.  For convenience, we restate it as follows.
\begingroup
\def\thetheorem{\ref{Fano}}
\begin{theorem}
Let $\mathcal{G}$ be a $\mathrm{PG}_2(2)$-free $3$-uniform hypergraph on $n$ vertices. Then for sufficiently large $n$, we have $q(\mathcal{G}) \leq q(\mathcal{B}_n)$, with equality if and only if $\mathcal{G}= \mathcal{B}_n$.
\end{theorem}
\addtocounter{theorem}{-1}
\endgroup
We establish a number of lemmas required for the proof of Theorem \ref {Fano}.
\begin{lemma}
\label{Bn lower bound}
For sufficiently large $n$, we have
\[
  q(\mathcal{B}_n) \geq \left\{
  \begin{array}{ll}
    \dfrac{3}{4} n^2 - \dfrac{3}{2} n & \text{if $n$ is even,} \\[4mm]
    \dfrac{3}{4} n^2 - \dfrac{3}{2} n - \dfrac{3}{4} + \dfrac{3}{2 n} & \text{if $n$ is odd.}
  \end{array}
  \right.
\]
\end{lemma}

\begin{proof}
Let $\mathbf{y} = (n^{-1/3}, n^{-1/3}, \dots, n^{-1/3}) \in \mathbb{R}^n$ be a unit vector.
By Theorem \ref{relaigh}, we have the lower bound
\[
  q(\mathcal{B}_n)
  \ge\sum_{e \in E(\mathcal{B}_n)} \bigg( \sum_{v \in e} y_v^3 + 3 y^e \bigg)= \binom{\lceil \frac{n}{2} \rceil}{2} \left\lfloor \frac{n}{2} \right\rfloor  \frac{6}{n} + \binom{\lfloor \frac{n}{2} \rfloor}{2} \left\lceil \frac{n}{2} \right\rceil \frac{6}{n}.
\]
We proceed by distinguishing two cases based on the parity of $n$.
If $n$ is even, this simplifies to
\[
q(\mathcal{B}_n) \geq   \frac{3}{4} n^2 - \frac{3}{2} n.
\]
If $n$ is odd, we have
\begin{align*}
    q(\mathcal{B}_n) &\geq \binom{(n+1)/2}{2} \frac{n-1}{2} \cdot \frac{6}{n} + \binom{(n-1)/{2}}{2} \frac{n+1}{2} \cdot \frac{6}{n} \\[2mm]
    & = \frac{3}{4} n^2 - \frac{3}{2} n - \frac{3}{4} + \frac{3}{2 n}.\qedhere
\end{align*}
\end{proof}

\medskip
Let $\mathfrak{B}_n$ denote the family of all $2$-colorable $3$-graphs on $n$ vertices. Suppose that $\mathcal{B}'_n \in \mathfrak{B}_n$ is a hypergraph that attains the maximum signless Laplacian spectral radius among $\mathfrak{B}_n$. Let $V(\mathcal{B}'_n)=V_1\cup V_2$ be a partition
such that every hyperedge of $\mathcal{B}'_n$ intersects both
 $V_1$ and $V_2$. Set $a=\mid V_1\mid$ and $b=\mid V_2\mid.$

\begin{lemma}\label{Bn upper bound}
    $q(\mathcal{B}'_n)\le  \frac{3}{4} n^2 - \frac{3}{2} n-(a-\frac{n}{2})^2.$
\end{lemma}
\begin{proof}
We begin the proof with a simple claim.
\begin{claim}\label{connect}
    $\mathcal{B}'_{n}$ is connected.
\end{claim}
\begin{proof}
    We proceed by contradiction. Suppose to the contrary that $\mathcal{B}'_{n}$ is disconnected. Let $B_{1}, B_{2}, \dots, B_{t}$ be the connected components of $\mathcal{B}'_n$, where $t \ge 2$. Let $\mathbf{x}=(x_{1}, x_{2}, \dots, x_{n})$ be a non-negative unit eigenvector of $\mathcal{Q}(\mathcal{B}_n')$ corresponding to  $q(\mathcal{B}'_{n})$. Let $u$ be the vertex such that $x_{u} = \max\{x_1, x_2, \dots, x_n\}$. Then $x_{u}>0$. Without loss of generality, we may assume that
    $q(\mathcal{B}'_{n})=q(B_1)$ and
     $u \in V(B_{1})$. Since $u$ is  not an isolated vertex, we have $|B_1|\ge 2$. Thus, we can choose a vertex $v\in V(B_1)$ such that $v \neq u$. Next, we pick an arbitrary vertex $w\in V(B_2)$. Let $\mathcal{B}^*_{n}$ be the hypergraph obtained from $\mathcal{B}'_{n}$ by adding the hyperedge $\{u,v,w\}$. Clearly, $\mathcal{B}^*_{n} \in \mathfrak{B}_{n}$. Then we have
    \[
        q(\mathcal{B}^*_{n}) - q(\mathcal{B}'_{n})
        \geq x^3_{u} + x^3_{v} + x^3_{w} + 3x_{u}x_{v}x_{w}
        \ge x_u^3 > 0,
    \]
which contradicts the maximality of $\mathcal{B}'_n$.
\end{proof}
\medskip
Since $\mathcal{B}'_n$ is connected, it has a principal eigenvector, denoted by $\mathbf{x} = (x_1, x_2, \dots, x_n)$.
\medskip
\begin{claim}\label{Fano1}
    $\mathcal{B}'_n$ must be a complete  $2$-colorable $3$-graph.
\end{claim}
\begin{proof}
Suppose that $\mathcal{B}'_n$ is not complete. Without loss of generality, we may assume  there are three vertices $v_1, v_2 \in V_1\text{ and } v_3\in V_2$ such that $v_1v_2v_3\notin E(\mathcal{B}'_n)$. Let $\mathcal{B}^*_n$ be the  hypergraph obtained
     by adding the hyperedge $\{v_1,v_2,v_3\}$ to $\mathcal{B}'_n$. Clearly, $\mathcal{B}^*_n\in \mathfrak{B}_n$. By Theorem \ref{relaigh}, we have
\begin{align*}
  q(\mathcal{B}^*_n)
  & \geq \sum_{e \in E(\mathcal{B}^*_n)} \Biggl( \sum_{v \in e} x_v^3 + 3 x^e \Biggr) = \sum_{e \in E(\mathcal{B}'_n)} \Biggl( \sum_{v \in e} x_v^3 + 3 x^e \Biggr) + \Biggl( \sum_{i\in [3]} x_{v_i}^3 + 3 \prod_{i\in [3]}x_{v_i} \Biggr).
\end{align*}
It then follows that
\begin{align*}
     q(\mathcal{B}^*_n) -q(\mathcal{B}'_n) \ge \Biggl( \sum_{i\in [3]} x_{v_i}^3 + 3 \prod_{i\in [3]}x_{v_i} \Biggr)>0,
\end{align*}
where the last inequality holds since $\mathbf{x}$ is the principal eigenvector.
This contradicts the maximality of $q(\mathcal{B}'_n)$.
\end{proof}

\medskip
By Claim \ref{Fano1}, $\mathcal{B}'_n$ is complete, so every vertex in the same class $V_i$ has the same degree. For convenience, we write $d(V_i)$ for the degree of the vertex in $V_i$ for each $i \in[2]$.
\medskip

\begin{claim}\label{Fano2}
For each  $i\in [2]$,  all entries of $\mathbf{x}$ belonging to $V_i$ are equal.
\end{claim}
\begin{proof}
The statement is trivial if $|V_i|=1$,. For the case $|V_i|\ge 2$, let $u$ and $v$ be two distinct vertices in $V_i$. Firstly, by applying  the eigenequation for $q(\mathcal{B}'_n)$ at the vertex $u$, we derive
\[
\big(q\left(\mathcal{B}'_n\right)-d\left(V_i\right)\big) x_u^2
=\sum_{e \in E_u} x^{e \setminus\left\{u\right\}}
=x_v \sum_{\substack{e \in E_u \cap E_v}} x^{e \setminus\{u, v\}}+\sum_{\substack{e \in E_u \setminus E_v}} x^{e \setminus\{u\}}.
\]
Similarly, invoking the eigenequation for $q(\mathcal{B}'_n)$ at the vertex $v$ yields
$$
\big(q\left(\mathcal{B}'_n\right)-d\left(V_i\right)\big) x_v^2
=\sum_{e \in E_v} x^{e \setminus\left\{v\right\}}
=x_u \sum_{\substack{e \in E_v \cap E_u}} x^{e \setminus\{u, v\}}+\sum_{\substack{e \in E_v \setminus E_ u}} x^{e \setminus\{v\}}.
$$

Let $\alpha_{3-i}$ be the sum of the entries of $\mathbf{x}$ in $V_{3-i}$. Note that
\begin{align}
\sum_{\substack{e \in E_v \cap E_u}} x^{e \setminus\{u, v\}}=\alpha_{3-i}.\label{eq:above}
\end{align}
Subtracting the first two equalities, and using \eqref{eq:above} together with the identity
 $\sum_{e \in E_u \setminus E_v} x^{e \setminus\{u\}}=\sum_{e \in E_v \setminus E_ u} x^{e \setminus\{v\}}$,  we obtain that
\[
    (x_u-x_v)\Big(\big(q(\mathcal{B}'_n)-d(V_i)\big)(x_u+x_v)+\alpha_{3-i}\Big)=0.
\]
Since $q(\mathcal{B}'_n)-d(V_i)$, $x_u+x_v$, and $\alpha_{3-i}$ are all positive, we obtain $x_u=x_v.$
\end{proof}

\medskip

In light of Claim \ref{Fano2},  let $x$ and $y$ be the entries of the eigenvector $\mathbf{x}$ corresponding to the vertices in $V_1$ and $V_2$, respectively. Define
$$C(a,b,x,y)=b\binom{a}{2} \big(2 x^3+y^3+3 x^2 y\big)+a\binom{b}{2} \left(2 y^3+x^3+3 x y^2\right)$$
where $a+b=n$ and $ax^3+by^3=1.$
Setting $u=x^3$ and $v=y^3$.  Then $au+bv=1$, and $C(a,b,x,y)$ can be rewritten as:
\[
C'(a, b, u, v)=b\binom{a}{2} \big(2 u+v+3 u^{\frac{2}{3}} v^{\frac{1}{3}}\big)+a\binom{b}{2} \big(2 v+u+3 u^{\frac{1}{3}} v^{\frac{2}{3}}\big),
\]
Clearly, $q(\mathcal{B}'_n)=C'(a,b,u,v).$ By introducing a parameter $t=\frac{n-5}{2n-4} (< \frac{1}{2})$, we deduce that
\begin{align*}
q(\mathcal{B}'_n)&= b\binom{a}{2} \bigg( 2 u+v+\frac{3}{(a^{\frac{2}{3}} b^{\frac{1}{3}})^{t+1}}(a^{t+1} u)^{\frac{2}{3}}(b^{t+1} v)^{\frac{1}{3}} \bigg) \\[2mm]
&\quad +a\binom{b}{2}  \bigg( 2 v+u+\frac{3}{(a^{\frac{1}{3}} b^{\frac{2}{3}})^{t+1}}(a^{t+1} u)^{\frac{1}{3}}(b^{t+1} v)^{\frac{2}{3}} \bigg) \\[2mm]
&\leq b\binom{a}{2}  \bigg( 2 u+v+\frac{1}{(a^{\frac{2}{3}} b^{\frac{1}{3}})^{t+1}}(2 a^{t+1} u+b^{t+1} v) \bigg) \\[2mm]
&\quad +a\binom{b}{2}  \bigg( 2 v+u+\frac{1}{(a^{\frac{1}{3}} b^{\frac{2}{3}})^{t+1}}(a^{t+1} u+2 b^{t+1} v) \bigg),
\end{align*}
where the inequality follows from Fact~\ref{fact6}.
Moreover, the equality holds if and only if $a^{t+1}u=b^{t+1}v$. Rearranging the terms with respect to $u$ and $v$, we have
\begin{align*}
q(\mathcal{B}'_n)\leq \,
  & \Biggl( b\binom{a}{2}  \biggl( 2 + 2 \Bigl( \frac{a}{b} \Bigr)^{\frac{t+1}{3}} \biggr) + a\binom{b}{2}  \biggl( 1 + \Bigl( \frac{a}{b} \Bigr)^{\frac{2(t+1)}{3}} \biggr) \Biggr) u \\[2mm]
  & + \Biggl( b\binom{a}{2}  \biggl( 1 + \Bigl( \frac{b}{a} \Bigr)^{\frac{2(t+1)}{3}} \biggr) + a\binom{b}{2}  \biggl( 2 + 2 \Bigl( \frac{b}{a} \Bigr)^{\frac{t+1}{3}} \biggr) \Biggr) v.
\end{align*}
Define
\begin{align*}
  R_1(a,b,t) &= \frac{1}{a} \Biggl(b\binom{a}{2} b\biggl( 2 + 2 \Bigl( \frac{a}{b} \Bigr)^{\frac{t+1}{3}} \biggr) +a\binom{b}{2} \biggl( 1 + \Bigl( \frac{a}{b} \Bigr)^{\frac{2(t+1)}{3}} \biggr) \Biggr), \\[2mm]
  R_2(a,b,t) &= \frac{1}{b} \Biggl(b\binom{a}{2}  \biggl( 1 + \Bigl( \frac{b}{a} \Bigr)^{\frac{2(t+1)}{3}} \biggr) + a\binom{b}{2}  \biggl( 2 + 2 \Bigl( \frac{b}{a} \Bigr)^{\frac{t+1}{3}} \biggr) \Biggr).
\end{align*}
By Fact \ref{yusuan}, we obtain \scalebox{0.95}{$q(\mathcal{B}'_n) \le \max \{ R_i(a,b,t) \mid i \in \{1,2\} \}$}. Noting that $R_1(a,b,t) =   R_2(b,a,t)$, it follows immediately that
$$q(\mathcal{B}'_n)\le\max\{  R_1(a,b,t),  R_1(b,a,t)\}.$$
We now proceed to compute calculate $R_1(a, b, t)$. Substituting $b = n-a$ into the expression for $R_1(a, b, t)$,  we deduce that
\begin{align*}
  R_1(a,b,t)
  &= \frac{n-a}{2} \Biggl( 2(a-1) \biggl( 1 + \Bigl( \frac{a}{n-a} \Bigr)^{\frac{t+1}{3}} \biggr) + (n-a-1) \biggl( 1 + \Bigl( \frac{a}{n-a} \Bigr)^{\frac{2(t+1)}{3}} \biggr) \Biggr) \\[2mm]
  &\le \frac{n-a}{2} \Biggl( 2(a-1) \biggl( 2 + \frac{t+1}{3} \Bigl( \frac{a}{n-a}-1 \Bigr) \biggr)\\
  &\   \ \ \ + (n-a-1) \biggl( 2 + \frac{2(t+1)}{3} \Bigl( \frac{a}{n-a}-1 \Bigr) \biggr) \Biggr) \\[2mm]
  &= (n-a)(n+a-3) + \frac{t+1}{3} (n-2)(2 a-n) \\[2mm]
  &= \frac{3}{4} n^2 - \frac{3}{2} n - \left( a - \frac{n}{2} \right)^2,
\end{align*}
where the inequality follows from Bernoulli's inequality, namely, $z^\lambda=\left(1+(z-1)\right)^\lambda\le 1+\lambda(z-1)$ for any $z>0, 0\le \lambda\le 1,$ and the last equality holds as $t = \frac{n-5}{2n-4}$.  By analogous reasoning, we  conclude that
$$R_1(b,a,t)\le   \frac{3}{4} n^2-\frac{3}{2} n-\left(b-\frac{n}{2}\right)^2= \frac{3}{4} n^2-\frac{3}{2} n-\left(a-\frac{n}{2}\right)^2.$$
This completes the proof.
\end{proof}

\medskip

\begin{lemma}\label{new}
   For sufficiently large $n$, we have $\mathcal{B}'_n=\mathcal{B}_n$.
\end{lemma}
\begin{proof}
We discuss two cases based on the parity of $n$. If $n$ is even, then by applying Lemmas \ref{Bn lower bound} and \ref{Bn upper bound}, we must have $a=n/2.$ So $\mathcal{B}'_n=\mathcal{B}_n$. If $n$ is odd, suppose for contradiction that $\mathcal{B}'_n\neq \mathcal{B}_n$. It then follows that
  $a\ge (n+3)/2$ or $a\le (n-3)/2$.  Using Lemmas \ref{Bn lower bound} and \ref{Bn upper bound}, we obtain
$$ \frac{3}{4} n^2 - \frac{3}{2} n - \frac{3}{4} + \frac{3}{2 n}\leq q(\mathcal{B}'_n)\leq \frac{3}{4} n^2-\frac{3}{2} n-\frac{9}{4},$$
which is  a contradiction. We therefore conclude that  $\mathcal{B}'_n=\mathcal{B}_n$ in both cases.
\end{proof}

Let $\mathcal{H}_n$ denote the family of all $n$-vertex $\mathrm{PG}_2(2)$-free $3$-graph with minimum degree greater than $(3/8-\varepsilon) n^{2}$.
Since $\varepsilon>0$ is sufficiently small, by Theorem \ref{Fano-deg-stability}, $\mathcal{H}_n\subseteq \mathfrak{B}_n$.
Note that $\mathcal{B}_n \in \mathcal{H}_n \subseteq \mathfrak{B}_n$. Combining this with Lemma \ref{new}, we have
$$q(\mathcal{B}_n) \leq q(\mathcal{H}_n)\leq q(\mathfrak{B}_n)= q(\mathcal{B}'_n) = q(\mathcal{B}_n).$$
So $ q(\mathcal{H}_n) = q(\mathcal{B}_n)$.

Now, we are ready to prove Theorem \ref{Fano}.

\medskip

\noindent\textbf{Proof of Theorem \ref{Fano}} Let $\varepsilon$ and $\sigma$ be sufficiently small numbers such that $\varepsilon\gg\sigma>0$ as described in Theorem \ref{criterion}. It suffices to show the two conditions are satisfied.
For the condition (\ref{condition 1}), given that  $\pi(\mathrm{PG}_2(2))=3/4$ and $r=3$, we derive \scalebox{0.95}{$\pi(\mathcal{F})n^{r-1}/(r-1)!=3n^2/8$}. A direct calculation  shows that
\begin{align}
 \left|\operatorname{ex}_3(n, \mathrm{PG}_2(2))-\operatorname{ex}_3(n-1, \mathrm{PG}_2(2))-\frac{3n^2}{8}\right|\le \sigma n^2.\notag
\end{align}
So the condition (\ref{condition 1}) of Theorem \ref{criterion} is satisfied.

Next, we verify that condition (\ref{condition 2}) of Theorem \ref{criterion}   holds.
A direct calculation yields
\begin{align*}
    \frac{6\mathrm{ex}_3(n, \mathrm{PG}_2(2))}{n} = \left\{
   \begin{array}{ll}
      \dfrac{3}{4} n^2 - \dfrac{3}{2} n & \text{if } n \text{ is even,} \\[3mm]
      \dfrac{3}{4} n^2 - \dfrac{3}{2} n - \dfrac{3}{4} + \dfrac{3}{2 n} & \text{if } n \text{ is odd.}
   \end{array}
   \right.
\end{align*}
By the fact $\mathcal{B}_n \in\mathcal{H}_n$ and Lemma \ref{Bn lower bound}, it suffices to verify that
\begin{align*}
   q(\mathcal{H}_n) - \frac{6\mathrm{ex}_3(n, \mathrm{PG}_2(2))}{n} \leq \sigma n.
\end{align*}
Note that
$$
   q(\mathcal{H}_n) -  \frac{6\mathrm{ex}_3(n, \mathrm{PG}_2(2))}{n} = q(\mathcal{B}_n) - \frac{6\mathrm{ex}_3(n, \mathrm{PG}_2(2))}{n}.
$$
Using Lemmas \ref{Bn upper bound} and \ref{new}, we obtain
$$
 q(\mathcal{B}_n) -  \frac{6\mathrm{ex}_3(n, \mathrm{PG}_2(2))}{n} \le \frac{3}{4} n^2 - \frac{3}{2} n - \frac{6\mathrm{ex}_3(n, \mathrm{PG}_2(2))}{n} \le \frac{3}{4} < \sigma n.
$$
Therefore, condition (\ref{condition 2}) of Theorem \ref{criterion} is satisfied.

 Let $\mathcal{G}$ be an $n$-vertex $\mathrm{PG}_2(2)$-free $3$-graph. By  Theorem \ref{criterion}, we have
\[
q(\mathcal{G}) \le q(\mathcal{H}_n)  = q(\mathcal{B}_n),
\]
 Moreover, equality $q(\mathcal{G}) = q(\mathcal{H}_n)$ holds only if $\mathcal{G} \in \mathcal{H}_n$.
Furthermore, since $\mathcal{B}_n$ is the unique maximizer in $\mathfrak{B}_n$, and given the  inclusion $\mathcal{B}_n \in \mathcal{H}_n \subseteq \mathfrak{B}_n$, it follows that  that $\mathcal{B}_n$ is also the unique maximizer in $\mathcal{H}_n$.
We thus conclude that $q(\mathcal{G}) \leq q(\mathcal{B}_n)$, with equality if and only if $\mathcal{G} = \mathcal{B}_n$. The proof is complete.\qed

\section{Concluding remarks}

In this paper, we establish a general criterion and apply it to the signless Laplacian spectral Tur\'an problem for the Fano plane.
Our findings indicate that any hypergraph Tur\'an problem that has the degree stability property and whose extremal constructions satisfy certain assumptions admits a corresponding
signless Laplacian spectral result.
Potential future research directions include applying this framework to other forbidden configurations such as expansions of complete graphs, generalized fans, and cancellative hypergraphs.

\end{document}